\documentclass[letterpaper,10pt,conference]{ieeeconf}

\usepackage{amsmath,amsfonts,amssymb,graphicx,graphics,enumerate,dsfont,float,color,bm}

\IEEEoverridecommandlockouts
\newcommand{\bbA}{\mathbb{A}}
\newcommand{\bbR}{\mathbb{R}}

\def\Xc{\mathcal{X}}

\def\Ac{\mathcal{A}}

\def\xh{\hat{x}}
\def\ah{\hat{a}}

\def\DelX0{\Delta(X_0)}
\def\delpi0{\delta^{\pi}}

\def\real{\mathbb{R}}
\newcommand{\argmax}{\mathop{\mathrm{argmax}}}
\newcommand{\maxm}{\mathop{\mathrm{maximize}}}
\newcommand{\sbjt}{\mathop{\mathrm{subject\ to}}}

\newtheorem{Theorem}{Theorem}

\newtheorem{Corollary}{Corollary}

\newtheorem{prop}{Proposition}

\newcommand{\comment}[1]{}

\title{\LARGE \bf Guaranteed Bounds for General Approximate Dynamic Programming}
\author{Yajing Liu, Edwin K. P. Chong, Ali Pezeshki, and Bill Moran%
\thanks{This work is supported in part by NSF under Grant CCF-1018472, and by AFOSR under Grant  FA9550-12-1-0418, and by CSU Information Science and Technology Center (ISTeC).}%
\thanks{Y. Liu is with the Department of Electrical and Computer Engineering, Colorado State University, Fort Collins, CO 80523, USA {\tt\small yajing.liu@ymail.com}}%
\thanks{E. K. P. Chong and A. Pezeshki are with the Department of Electrical and Computer Engineering, and the Department of Mathematics, Colorado State University, Fort Collins, CO 80523, USA {\tt\small Edwin.Chong,Ali.Pezeshki@Colostate.Edu}}%
\thanks{B. Moran is with the Department of Electrical and Electronic Engineering, The University of Melbourne, Melbourne, VIC 3010, Australia {\tt\small wmoran@unimelb.edu.au}}%
}

\begin{document}

\maketitle
\thispagestyle{empty}
\pagestyle{empty}

\begin{abstract}

In this paper, we will develop a systematic approach to deriving guaranteed bounds for approximate dynamic programming (ADP) schemes in optimal control problems. Our approach is inspired by our recent results on bounding the performance of greedy strategies in optimization of string-submodular functions over a finite horizon. The approach is to derive a string-submodular optimization problem, for which the optimal strategy is the optimal control solution and the greedy strategy is the ADP solution. Using this approach, we show that any ADP solution achieves a performance that is at least a factor of $\beta$ of the performance of the optimal control solution, which satisfies Bellman's optimality principle. The factor $\beta$ depends on the specific ADP scheme, as we will explicitly characterize. To illustrate the applicability of our bounding technique, we present examples of ADP schemes, including the popular rollout method.

\end{abstract}

\section{Introduction}

In sequential decision making, adaptive sensing, and adaptive control, we are frequently faced with optimally choosing a string (finite sequence) of actions over a finite horizon to maximize an
objective function. However, computing the optimal strategy (optimal sequence of actions) is often difficult. One approach is to use dynamic programming via Bellman's principle for optimality (see, e.g.,
\cite{Bertsekas2000}). However, the computational complexity of this approach grows exponentially with the size of the action space and the decision horizon. Because of this inherent complexity, for years, there has been interest in developing approximation methods for solving dynamic programming problems. Although a wide range of approximate dynamic programming (ADP) methods have been developed (see, e.g., \cite{Pow07}), a general systematic technique to provide performance guarantees for them has remained elusive. In this paper, we will develop a systematic approach to deriving guaranteed bounds for ADP schemes. Our approach is inspired by our recent results in \cite{ZhC13J} and  \cite{ZhW13}) on bounding the performance of greedy strategies in optimization of string-submodular functions.

Submodularity of functions over finite sets plays an important role
in discrete optimization (see, e.g., \cite{nemhauser1978best}, \cite{nemhauser1978analysis}, \cite{conforti1984submodular}, \cite{sviridenko2004note}, \cite{streeter2008online}, \cite{alaei2010maximizing}, \cite{vondrak2008optimal}, \cite{vondrak2010submodularity},
\cite{feige2006approximation}, \cite{feige2010submodular}, \cite{shamaiah2010greedy}, \cite{wang2012}, \cite{ageev2004pipage}, \cite{kulik2009maximizing}, and \cite{lee2010submodular}). It has been shown that, under submodularity, the greedy strategy provides at least a
constant-factor approximation to the optimal strategy. For example,
the celebrated result of Nemhauser \emph{et al.}~\cite{nemhauser1978analysis} states
that for maximizing a monotone submodular function $F$ over a
uniform matroid such that $F(\emptyset)=0$ (here $\emptyset$ denotes
the empty set), the value of the greedy strategy is no less than a
factor $(1-e^{-1})$ of that of the optimal strategy. This is a
powerful result. But a drawback is that submodular functions studied in
most previous papers are defined on the power set of a given finite
set. In contrast, in adaptive control and sensing, we are interested in choosing a string of action sequentially, and the value of the objective function depends on the \emph{order} of
these actions. In consequence, we cannot apply the result of
Nemhauser \emph{et al.}~\cite{nemhauser1978analysis} or its related results on submodularity over finite sets.

To compare the greedy and optimal strategies for functions defined over strings, in \cite{ZhC13J} and \cite{ZhW13}, we have introduced the notion of \emph{string}-submodularity,
which builds on the notion of set-submodularity in combinatorial
optimization. We have shown that, under string-submodularity,
any greedy strategy is suboptimal by a factor of at worst $(1-e^{-1})$, entirely
consistent with the result of Nemhauser \emph{et al.}~\cite{nemhauser1978analysis}.
Our framework also includes characterizing the \emph{curvature} of
string-submodular functions, which roughly corresponds to the
quantitative ``degree'' of submodularity. In fact, there are several
notions of curvature (to be described later). Subject to curvature,
we have derived suboptimality bounds for greedy strategies that are
strictly better than $(1-e^{-1})$. These results represent the
state-of-the-art in bounding greedy strategies in string-submodular
optimization problems.

In this paper, inspired by the bounding techniques in \cite{ZhC13J} and \cite{ZhW13}, we develop the first systematic approach to deriving performance bounds for general ADP methods for optimal control problems. To set up our approach, in Section \ref{sc:II}, we review our string-submodularity results, notions of curvature, and  the corresponding bounds from. In Section \ref{sc:III}, we first describe a general optimal control problem and a class of ADP schemes for approximating optimal control solutions. We then describe our approach to bounding the performance of such ADP schemes. The idea is to define a string-submodular optimization problem for which the optimal strategy is the optimal control solution, and the greedy strategy is the ADP solution. Though, inspired by our previous work, the bounding of ADP schemes is based on a new technique for general string-optimization problems. The results in Section \ref{sc:II}, simply set the stage and terminology for new developments and results that we will present in Section \ref{sc:IV}. We show that any ADP solution achieves a performance that is at least a factor of $\beta$ of the performance of the optimal control solution (satisfying Bellman's optimality principle). The factor $\beta$ depends on the specific ADP scheme, in way that we will explicitly characterize. In Section \ref{sc:V}, we present a few examples of ADP schemes to illustrate the application of our results. In particular, we consider rollout policies which represent a well-studies family of ADP schemes (see, e.g., \cite{BeT97}). Finally, in Section \ref{sc:VI}, we present our concluding remarks.

\section{String-Submodularity and Performance Bounds for Greedy Strategies}\label{sc:II}

In this section, we review our string-submodularity results, notions of curvature, and the corresponding bounds from \cite{ZhC13J} and \cite{ZhW13}. These results show that greedy strategies for optimizing a string-submodular function achieve at least a factor of $\alpha$ of the performance of optimal strategies, which are characterized by Bellman's optimality principle. The factor $\alpha$ depends on the specific objective function to be optimized and its various curvatures, but it is at least $(1-e^{-1})$. The results presented here set the stage, terminology, and the inspiration for our new developments in Section \ref{sc:IV} for bounding the performance of ADP schemes.

\subsection{String-Submodularity and Curvatures}

Let $\bbA$ be a set of possible actions. At each stage $i$, we
choose an action $a_i$ from $\bbA$. Let $A=(a_1,a_2,\ldots,a_k)$
be a \emph{string} of actions taken over $k$ consecutive
stages, where $a_i\in \bbA$ for $i=1,2,\ldots, k$. Let $\bbA^*=\{(a_1,a_2,\ldots,a_k)| \ k=0,1,\ldots \textrm{and} \  a_i\in \bbA,\text{ }  i=1,2\ldots, k\}$ be the set of all possible strings of actions. Note that $k=0$ corresponds to the empty string (no action taken), denoted by $\varnothing$.

For a given string $A=(a_1,a_2,\ldots,a_k)$, we define its \emph{string length} as $k$, denoted $|A|=k$. If $M=(a_1^m,a_2^m,\ldots, a_{k_1}^m)$ and $N=(a_1^n,a_2^n,\ldots, a_{k_2}^n)$ are two strings in $\bbA^*$, we say $M=N$ if $|M|=|N|$ and $a_i^m=a_i^n$ for each $i=1,2,\ldots, |M|$. Moreover, we define string \emph{concatenation} as $M\oplus N= (a_1^m,a_2^m,\ldots, a_{k_1}^m,a_1^n,a_2^n,\ldots, a_{k_2}^n)$. If $M$ and $N$ are two strings in $\bbA^*$, we write $M\preceq N$ if we have $N=M\oplus L,$ for some $L\in \bbA^*$. In other words, $M$ is a \emph{prefix} of $N$.

\noindent\textbf{String Submodularity.} A function from strings to real numbers, $f: \bbA^*\to \bbR$, is \emph{string submodular} if
\begin{itemize}
\item[i.] $f$ has the \emph{forward-monotone} property, i.e., $ \forall M\preceq N \in \bbA^*,$  $f(M)\leq f(N)$.
\item[ii.] $f$ has the \emph{diminishing-return} property, i.e., $\forall M\preceq N \in \bbA^*, \forall a\in \bbA$, $f(M\oplus (a))-f(M) \geq f(N\oplus (a))-f(N)$.
\end{itemize}
We assume, without loss of generality, that  $f(\varnothing)=0$. Otherwise, we can replace $f$ with the marginalized function $f-f(\varnothing)$. From the forward-monotone property, we know that $f(M) \geq 0$ for all $M\in \bbA^*$.

\noindent\textbf{Curvatures.} We define several notions of curvature for $f$ as follows.
\begin{enumerate}
\item \emph{Total backward curvature} of $f$:
\begin{align}
\sigma=\max_{a \in \bbA, M\in \bbA^*}\left\{1-\frac{f( (a)\oplus M)-f( M)}{f((a)) -f(\varnothing)}\right\}.\nonumber
%\label{con3}
\end{align}
\item \emph{Total backward curvature of $f$ with respect to string $M\in \bbA^*$}:
\begin{align}
\sigma(M)=\max_{N\in \bbA^*, 0<|N|\leq K}\left\{1-\frac{f(N\oplus M)-f(M)}{f(N) -f(\varnothing)}\right\}.\nonumber
%\label{con4}
\end{align}
\item \emph{Total forward curvature of $f$:}
\begin{align}
\epsilon=\max_{a \in \bbA, M\in \bbA^*}\left\{1-\frac{f( M\oplus (a))-f( M)}{f((a)) -f(\varnothing)}\right\}.\nonumber
\end{align}
\item \emph{Total forward curvature of $f$ with respect to $M$}:
\begin{align}
\epsilon(M)=\max_{N\in \bbA^*, 0<|N|\leq K}\left\{1-\frac{f(M\oplus N)-f(M)}{f(N) -f(\varnothing)}\right\}.\nonumber
\end{align}
\item \emph{Elemental forward curvature of $f$}:
\begin{align}
\eta=\max_{a_i, a_j \in \bbA, M\in \bbA^*}{\frac{f(M\oplus (a_i)\oplus (a_j))-f(M\oplus (a_i))}{f(M\oplus (a_j)) -f(M)}}.\nonumber
\end{align}
%\item[(C6)] \emph{$K$-elemental forward curvature of $f$}:
%\begin{align}
%\hat \eta &=\max_{a_i,a_j\in \bbA,M\in \bbA^*,|M|\leq 2K-2}  \frac{f(M\oplus (a_i) \oplus (a_j))-f(M\oplus (a_i))}{f(M \oplus (a_j))-f(M)}.\nonumber
%%\label{curvature}
%\end{align}
\end{enumerate}

\subsection{Performance Bounds for Greedy Strategies}

Consider the problem of finding a string $M\in \bbA^*$, with a length $|M|$ not larger than $K$ (prespecified), to maximize the objective function $f$, that is
\begin{align}\label{eqn:1}
\begin{array}{l}
\text{maximize} \ \    f(M) \\
\text{subject to} \ \ M\in\bbA^*, |M|\leq K.
\end{array}
\end{align}

We define optimal and greedy strategies for \eqref{eqn:1} as follows:
\begin{itemize}
\item[(1)] \emph{Optimal strategy}: Consider the problem \eqref{eqn:1} of finding a string that maximizes $f$ under the constraint that the string length is not larger than $K$. We call a solution of this problem an \emph{optimal strategy} (a term we already have used repeatedly before).
Note that if the function $f$ is forward monotone and there exists an optimal strategy, then there exists one with length $K$.
\item[(2)] \emph{Greedy strategy}:
A string $G_{k}=(g_1,g_2,\ldots,g_{k})$ is called \emph{greedy} if $\forall i=1,2,\ldots,k,$
\begin{align*}
g_i&\in\mathop{\argmax}\limits_{g\in \bbA} f((g_1,g_2,\ldots,g_{i-1},g)),
\end{align*}
where $\argmax$ denotes the set of actions that maximize $f((g_1,g_2,\ldots,g_{i-1},g))$. 
\end{itemize}

Let $I$ be the subset of $\bbA^*$ with maximal string length $K$: $I=\{A\in \bbA^*: |A|\leq K\}.$
We call $I$ a \emph{uniform structure}. Note that the way we define uniform structures is similar to the way independent sets associated with uniform matroids are defined. We now present the relationship between total curvatures and approximation bounds for the greedy strategy.

\begin{Theorem} \cite{ZhC13J} (Greedy approximation bounds involving total curvatures). Consider a string submodular function $f$. Let $O$ be a solution to \eqref{eqn:1}. Then, any greedy string $G_K$ satisfies
\begin{itemize}
\item[(i)]
\begin{align*}
f(G_K) &\geq \frac{1}{\sigma(O)}\left(1-\left(1-\frac{\sigma(O)}{K}\right)^{K}\right)f(O)\\
&> \frac{1}{\sigma(O)}(1-e^{-\sigma(O)})f(O),
 \end{align*}
\item[(ii)]  \[f(G_K) \geq (1-\max_{i=1,\ldots, K-1} \epsilon(G_i)) f(O).\]
\end{itemize}\label{thm1}
\end{Theorem}

Under the framework of maximizing submodular set functions, similar results are reported in \cite{conforti1984submodular}. However, the forward and backward algebraic structures are not exposed in \cite{conforti1984submodular} because the total curvature there does not depend on the order of the elements in a set. In the setting of maximizing string submodular functions, the above theorem exposes the roles of forward and backward
algebraic structures in bounding the greedy strategy.

The results in Theorem 1 imply that for a string submodular function, we have $\sigma(O) \geq 0$. Otherwise, part (i) of Theorem~1 would imply that $f(G_K)\geq f(O)$, which is absurd. Moreover, recall that if the function is backward monotone, then $\sigma(O)\leq \sigma \leq 1$ and we have the following result.

\begin{Corollary} \cite{ZhC13J} (Universal greedy approximation bounds involving total curvatures). Suppose that $f$ is string-submodular and backward monotone. Then,
\begin{itemize}
\item[(i)]
\begin{align*}
f(G_K) & \geq \frac{1}{\sigma}\left(1-\left(1-\frac{\sigma}{K}\right)^{K}\right)f(O) \\
& > \frac{1}{\sigma}(1-e^{-\sigma})f(O),
 \end{align*}
\item[(ii)]
\[f(G_K) \geq (1- \epsilon) f(O).\]
\end{itemize}\label{cor1}
\end{Corollary}

Note that the bounds $\frac{1}{\sigma}(1-e^{-\sigma})$ and $(1-\epsilon)$ are independent of the length constraint $K$. Therefore, the above bounds are universal lower bounds for the greedy strategy for all possible length constraints.  Part (i) of Corollary~\ref{cor1} implies that in the backward monotone case, where $\sigma\le 1$, any greedy string $G_K$ satisfies the universal bound $f(G_K) > (1-e^{-1})f(O)$.

\begin{Theorem} \cite{ZhC13J} (Greedy approximation bounds involving elemental curvature). Consider a forward-monotone function $f$ with elemental forward curvature $\eta$. Let $O$ be an optimal solution to \eqref{eqn:1}. Suppose that $f(G_{i}\oplus O) \geq f(O)$ for $i=1,2,\ldots, K-1$. Then, any greedy string $G_K$ satisfies
%\[
%f(G_K) \geq f(O) \left(1-(1-\frac{1}{K_{\hat \eta}})^{K}\right)
% \geq f(O)\left(1-(1-\frac{1}{K_{\eta}})^{K}\right),
%\]
\[
f(G_K) \geq  f(O)\left(1-(1-\frac{1}{K_{\eta}})^{K}\right),
\]
where $K_{\eta}= ({1-\eta^{K}})/({1-\eta})$ if $\eta\neq 1$ and $K_{\eta}=K$ if $\eta=1$.
\label{thm2}
\end{Theorem}

Recall that $\eta$ does not depend on the length constraint $K$. Therefore, the lower bound using $K_{\eta}$ is a universal lower bound for the greedy strategy. Now suppose that $f$ is string submodular. Then, we have $\eta \leq 1$. Because $1-(1-\frac{1}{K_{\eta}})^{K}$ is decreasing as a function of $\eta$, an immediate consequence of Theorem~\ref{thm2} is that any greedy string $G_K$ satisfies the universal bound $f(G_K) > (1-e^{-1})f(O)$.

\subsection{Other Results}

In the previous section, we considered the case where $I$ is a uniform structure. In \cite{ZhC13J} and \cite{ZhW13}, we have also studied the case where $I$ is a \emph{non-uniform} structure, by introducing the notion of string-matroid, and have derived bounds that quantify the performance of greedy strategies relative to optimal strategies in terms of various curvatures of the objective function. We leave these results out for the sake of brevity and refer the reader to \cite{ZhC13J} and \cite{ZhW13} for details.

A number of other researchers (see \cite{streeter2008online}, \cite{alaei2010maximizing}, and \cite{golovin2011adaptive}) have also considered bounding the performance of greedy strategies using extensions of set submodularity to string-submodularity. In particular, Streeter and Golovin~\cite{streeter2008online} showed that if the function $f$ is \emph{forward} and \emph{backward} monotone: $f(M\oplus N) \geq f(M)$ and $f(M\oplus N) \geq f(N)$ for all $M,N\in \bbA^*$, and $f$ has the diminishing-return property: $f(M\oplus (a))-f(M)\geq f(N\oplus (a))-f(N)$ for all $a\in \bbA$, $M,N\in \bbA^*$ such that $M$ is a prefix of $N$, then the greedy strategy achieves at least a $(1-e^{-1})$-approximation of the optimal strategy. However, the notions of string submodularity and various curvature that we have introduced in our recent work \cite{ZhC13J}, \cite{ZhW13} provide us with weaker sufficient conditions under which the greedy strategy still achieves at least a $(1-e^{-1})$-approximation of the optimal strategy.

\section{Bounding ADP Schemes in Optimal Control}\label{sc:III}

In this section, we first describe a general optimal control problem and a class of ADP schemes for approximating optimal control solutions. We then describe our approach to bounding the performance of such ADP schemes.

\subsection{General Optimal Control Problems}

To begin our formulation of a general optimal control problem, let $\Xc$ denote a set of states and $\Ac$ a set of control actions.
Given $x_1\in\Xc$ and functions $h_k:\Xc\times\Ac\to\Xc$ and $r_k:\Xc\times\Ac\to\real_+$ for $k=1,\ldots,K$, consider the optimal control problem
\begin{equation}
\begin{aligned}
\maxm_{a_1,\ldots,a_K\in\Ac} &\quad  \sum_{k=1}^K r_k(x_k,a_k) \\
\sbjt &\quad x_{k+1} = h_k(x_k,a_k),\ k=1,\ldots,K-1.
\end{aligned}
\label{eqn:2}
\end{equation}

Think of $a_k$ as the \emph{control action} applied at time $k$ and $x_k$ the
\emph{state} visited at time $k$. The real number $r_k(x_k,a_k)$ is the \emph{reward} accrued at time $k$ by applying $a_k$ at state $x_k$. This form of optimal control problem covers a wide variety of optimization problems found in many areas, ranging from engineering to economics. In particular, many adaptive sensing problems have this form (see, e.g., \cite{LiC12}).

The solution to the optimal control problem above is characterized by Bellman's principle of dynamic programming. To explain, for each $k=1,\ldots,K$, define functions
$V_k:\Xc\times\Ac^{K-k+1}\to\real_+$ by
\[
V_k(x_k,(a_k,\ldots,a_K)) = \sum_{i=k}^K r_i(x_i,a_i)
\]
where $x_{i+1} = h_i(x_i,a_i)$, $i=k,\ldots,K-1$.
The optimal control problem can be written as
\begin{align*}
\maxm_{a_1,\ldots,a_K\in\Ac} &\quad  V_1(x_1,(a_1,\ldots,a_K)) \\
\sbjt &\quad x_{k+1} = h_k(x_k,a_k),\ k=1,\ldots,K-1,
\end{align*}
wher $x_1$ is given. Let $o_1,\ldots,o_K$ be an optimal solution to this problem, and given $x_1$, define $x_1^*=x_1$ and $x_{k+1}^*=h_k(x_k^*,o_k)$, $k=1,\ldots,K-1$. This is the sequence of states visited as a result of the optimal control actions $o_1,\ldots,o_K$. Then, Bellman's principle states that for
$k=1,\ldots,K$, we have
\begin{equation}
\begin{array}{lr}
V_k(x_k^*,(o_k,\ldots,o_K))=&\\
 \mathop{\max}\limits_{a\in\Ac}
\{r_k(x_k^*,a) + V_{k+1}(h_k(x_k^*,a),(o_{k+1},\ldots,o_K))\},&\\
o_k \in & \\
\mathop{\argmax}\limits_{a\in\Ac}
\{r_k(x_k^*,a) + V_{k+1}(h_k(x_k^*,a),(o_{k+1},\ldots,o_K))\}, &
\end{array}\label{eqn:3}
\end{equation}
with the convention that $V_{K+1}(\cdot)\equiv 0$. Moreover, any sequence of control actions satisfying (\ref{eqn:3}) above is optimal. The term $V_{k+1}(h_k(x_k^*,a),(o_{k+1},\ldots,o_K))$ is called the \emph{value-to-go} (VTG).

Bellman's principle provides a method to compute an optimal solution: We use (\ref{eqn:3}) to iterate backwards over the time indices $k=K,K-1,\ldots,1$, keeping the states as variables, working all the way back to $k=1$. This is the familiar \emph{dynamic programming algorithm}. However, the procedure suffers from the \emph{curse of dimensionality} and is therefore impractical for many problems of interest: merely storing the iterates $V_k(\cdot, (o_k,\ldots,o_K))$ requires an exponential amount of memory.  Therefore, designing computationally tractable approximation methods remains a topic of active research.

\subsection{ADP Schemes}

A well-studied class of \emph{approximate dynamic programming} (ADP) approaches rests on approximating the VTG $V_{k+1}(h_k(x_k^*,a),(o_{k+1},\ldots,o_K))$  by some other term $W_{k+1}(\xh_k,a)$. In this method, we start at time $k=1$, at state $\xh_1=x_1$, and for each $k=1,\ldots,K$, we compute the control action and state using
\begin{equation}
\begin{array}{rl}
\ah_k \hspace{-.2cm}& \in \mathop{\argmax}\limits_{a\in\Ac} \{r_k(\xh_k,a) + W_{k+1}(\xh_k,a)\}, \\
\xh_{k+1} \hspace{-.2cm}& =  h_k(\xh_k,\ah_k).
\end{array}\label{eqn:4}
\end{equation}
The VTG approximation $W_{k+1}(\xh_k,a)$ can be based on a number of methods, ranging from heuristics to reinforcement learning \cite{SuB98} to rollout \cite{BeT97}.

A natural question is ``what is the performance of the ADP approach above relative to the optimal solution?'' The answer, of course, depends on the specific VTG approximation. If the VTG approximation is equal to the true VTG, then the procedure above generates an optimal solution. In general, the procedure produces something suboptimal. But how suboptimal? This question has alluded general treatment but has remained an issue of great interest to designers and users of ADP methods. In the following section, we develop a systematic approach to answering this fundamental
question.

\subsection{Deriving Performance Bounds for ADP Schemes}

We now describe our approach to bounding the performance of such ADP schemes. The idea is to define a string-submodular optimization problem for which the optimal strategy is the optimal control solution, and the greedy strategy is the ADP solution. Though inspired by our previous work (reviewed in Section \ref{sc:II}), the bounding of ADP schemes is based on a new technique for general string-optimization problems.

To see how our approach works, let $\Ac_k$ be the set of all strings of symbols in $\Ac$ with length not exceeding $k$. Define the function $f:\Ac_K\to\mathbb{R}_+$ by
$f(\varnothing)=0$ and
\[f((a_1,a_2,\ldots,a_k)) = \sum_{i=1}^k r_i(x_i,a_i) + W_{k+1}(x_k,a_k),
\]
for $k=1,\ldots,K$, where $x_{k+1}=h_k(x_k,a_k)$ as before and $W_{K+1}(\cdot)\equiv 0$
by convention. Using this string function $f$, we can now define the
optimization problem of finding a string $(a_1,\cdots, a_K)$ to maximize
$f((a_1,\ldots,a_K))$. This is an instance of the string-optimization problem described earlier.

It is clear that $f((a_1,\ldots,a_K)) = \sum_{i=1}^K r_i(x_i,a_i)$, which is the objective function in (\ref{eqn:2}). Hence, the string-optimization problem defined above is equivalent to the optimal control problem (\ref{eqn:2}). Next, notice that a greedy scheme by definition has the following form, given $(g_1,\ldots,g_{k-1})$:
\begin{equation}
\begin{aligned}
g_k &\in \mathop{\argmax}\limits_{g\in\Ac} f((g_1,\ldots,g_{k-1},g)) \\
&\in \argmax_{g\in\Ac}\{\sum_{i=1}^{k-1} r_i(x_i,g_i) + r_k(x_k,g) + W_{k+1}(x_k,g)\}\\
%& - (\sum_{i=1}^{k-1} r_i(x_i,g_i) + W_{k}(x_{k-1},g_{k-1}))\}\\
&\in \argmax_{g\in\Ac} \{r_k(x_k,g) + W_{k+1}(x_k,g)\}.
\end{aligned}
\label{eqn:5}
\end{equation}
This is simply the ADP scheme in (\ref{eqn:4}). Hence, we have the following result.
\begin{prop}
The ADP scheme in (\ref{eqn:4}) is a greedy strategy for the string-optimization problem
\begin{align}
\begin{array}{l}
\text{maximize} \ \    f((a_1,a_2,\ldots,a_K)) \\
\text{subject to} \ \ (a_1,a_2,\ldots,a_K)\in\Ac_K.
\end{array}
\end{align}

\end{prop}

Using this proposition, we can show that any ADP solution achieves a performance that is at least a factor of $\beta$ of the performance of the optimal control solution (satisfying Bellman's optimality principle). The factor $\beta$ depends on the specific ADP scheme as we will explicitly show in the next section.

%\begin{Remark}
%When $W_{k+1}(x_k,g)=\sum_{i=k+1}^Kr_i(\hat{x}_i,\pi_b(\hat{x}_i))$, where $\hat{x}_{k+1}=h_k({x}_k,g)$, $\hat{x}_{i+1}=h_k(\hat{x}_i, \pi_b(\hat{x}_i))$ for $ k+1\leq i\leq K$,   and $\pi_b:\Xc\rightarrow \Ac$, the ADP method is called a rollout algorithm (see, e.g., \cite{BeT97}). In Section \ref{sc:V}, we apply our bounding technique to rollout schemes.
%\end{Remark}

\section{Main Results}\label{sc:IV}

\subsection{General Bound}

In the last section, we introduced a general optimal control problem and an associated class of ADP schemes. We then formulated a string-optimization problem associated with a given optimal control problem and ADP scheme with the property that any optimal strategy for the string-optimization problem is an optimal control solution and any greedy strategy is the ADP solution. This allows us to use bounding methods for greedy strategies for string-optimization to derive bounds for ADP methods. However, it turns out that the results in Section~\ref{sc:II} do not directly apply to the string-optimization problem we formulated in Section~\ref{sc:III}. More specifically, the function $f$ in Section~\ref{sc:III} is defined only on $\Ac_K$ (i.e., strings of length at most $K$), whereas the results in Section~\ref{sc:II} require $f$ to be defined on strings of length greater than $K$. To address this issue, we now present a \emph{new} result for bounding greedy strategies for string-optimization problems.

Let $f:\Ac_K\to\real_+$ be an objective function. Consider the optimization problem
\begin{equation}
\label{eqn:max}
\text{maximize  }  f(S) \ \   \text{subject to} \ \ S\in\Ac_K, |S|=K.
\end{equation}
Let $O_K=(o_1,\ldots,o_K)$ be optimal for (\ref{eqn:max}). 
Let $G_K=(g_1, \ldots, g_K)$ be a greedy strategy for (\ref{eqn:max}), defined as before: given $g_1, \cdots, g_{k-1}$, 
\begin{equation}
\label{eqn:greedy}
g_k\in\mathop{\argmax}_{g\in \mathbb{A}} f((g_1,\cdots, g_{k-1}, g)).
\end{equation}
As before, write $G_0=O_0=\varnothing$ and
for $k=1,\ldots,K$, 
$G_k=(g_1,\ldots,g_k)$ and $O_k=(o_1,\ldots,o_k)$.

Inspired by the results in Section~\ref{sc:II}, define the 
\emph{forward curvature of $f$ with respect to $G_k$} by 
\begin{equation}
\epsilon_k=1-\frac{f(G_{k+1})-f(G_k)}{f((g_1))-f(\varnothing)}, \quad 0\leq k\leq K-1.
\label{eqn:cur1}
\end{equation}
Notice that $\epsilon_0 = 0$.
Next, define the \emph{elemental forward curvature of $f$ with respect to $O_k$} by 
\begin{equation}
\eta_k=\frac{f(O_{k+1})-f(O_{k})}{f(o_{k+1})-f(\varnothing)}, \quad  0\leq k\leq K-1.
\label{eqn:cur2}
\end{equation}
Notice that $\eta_0 = 1$.
We now present a result that bounds $f(G_K)$ relative to $f(O_K)$, using the definitions above.

\begin{Theorem}
\label{thm:bounds}
The following bound holds:
$f(G_K)\geq \beta f(O_K)$, where
\[
\beta = \frac{\sum_{i=0}^{K-1}(1-\epsilon_i)}{\sum_{i=0}^{K-1}\eta_i}.
\]
\end{Theorem}

\begin{proof}
Using the definition of the forward curvature of $f$ with respect to $G_k$, we have
\begin{align*}
f(G_2)-f(G_1) &= (1-\epsilon_1)f(G_1),\\
f(G_3)-f(G_2) &= (1-\epsilon_2)f(G_1),\\
&\vdots\\
f(G_k)-f(G_{k-1}) &= (1-\epsilon_{k-1})f(G_1).\\
&\vdots\\
f(G_K)-f(G_{K-1}) &= (1-\epsilon_{K-1})f(G_1),
\end{align*}
which leads to
\begin{equation}
f(G_K)=\sum_{i=0}^{K-1}(1-\epsilon_i) f(G_1).
\label{eqn:GK1}
\end{equation}
By the definition of elemental forward curvature of $f$ with respect to $O_k$, we have 
\begin{align}
f(O_K)
&=\sum_{i=1}^{K}(f((o_1,\cdots,o_i))-f( (o_1,\cdots,o_{i-1})))\nonumber\\
&=\eta_0f(o_1)+\eta_1 f(o_2)+\cdots+\eta _{K-1}f(o_K)\nonumber\\
&\leq\sum_{i=0}^{K-1}\eta_i f(G_1).\nonumber
\end{align}
%Because $f(O_K)\geq f(G_1)>0,$ we have $\sum_{i=0}^{K-1}\eta_i \geq 1$. 
where $f(G_1)\ge f(a)$ for any $a\in\Ac$ by \eqref{eqn:greedy}. Therefore, 
\begin{equation}
 f(G_1)\geq  \frac{1}{\sum_{i=0}^{K-1}\eta_i}f(O_K).
\label{eqn:G1O}
\end{equation}
Combining (\ref{eqn:GK1}) and (\ref{eqn:G1O}), we get 
\[
f(G_K) \geq \frac{\sum_{i=0}^{K-1}(1-\epsilon_i)}{\sum_{i=0}^{K-1}\eta_i}f(O_K)
\]
as desired.
\end{proof}

\emph{Remarks:}
\begin{itemize}
\item[1.] Notice that the bound above holds without any assumption on the monotonicity of $f$. However, the bound is only meaningful if $\beta\geq 0$. A sufficient condition for this is the monotonicity of $f$. More precisely, if $f$ is forward monotone with respect to $G_k$, then $\epsilon_k\leq 1$ for each $k$, and $\epsilon_0=0$, in which case $\beta> 0$.
\item[2.] It is easy to check that if 
\[
\sum_{i=0}^{K-1} \epsilon_i + \eta_i \leq K,
\]
then $f(G_K) = f(O_K)$; i.e., the greedy strategy is optimal.
\end{itemize}

\subsection{Bounding ADP Schemes}

We can now apply the result of Theorem~\ref{thm:bounds} to the function $f$ defined in Section~\ref{sc:III}. Doing so will provide bounds on general ADP schemes relative to optimal control solutions. To begin, recall that
\[
f((a_1,\ldots,a_k)) = \sum_{i=1}^k r_i(x_i,a_i) + W_{k+1}(x_k,a_k).
\]
Assume without loss of generality that $f$ is a nonnegative function (for otherwise, we can simply add a constant to each $W_{k+1}$ term).
For this form of $f$, we have
%\begin{equation*}
%\begin{array}{ll}
%\epsilon_k=1-&\\
%\frac{r_{k+1}(x_{k+1}, g_{k+1})+W_{k+2}(x_{k+1}, g_{k+1})-W_{k+1}(x_k, g_k)}{r_1(x_1, g_1)+W_2(x_1,g_1)},&
%\end{array} 
%\end{equation*}
\begin{align*}
&\epsilon_k=1-\\
&\frac{r_{k+1}(x_{k+1}, g_{k+1})+W_{k+2}(x_{k+1}, g_{k+1})-W_{k+1}(x_k, g_k)}{r_1(x_1, g_1)+W_2(x_1,g_1)},
\end{align*}
and 
\begin{align*}
&\eta_k=\\
&\frac{r_{k+1}(x_{k+1}, o_{k+1})+W_{k+2}(x_{k+1}, o_{k+1})-W_{k+1}(x_k, o_k)}{r_1(x_1, o_{k+1})+W_2(x_1, o_{k+1})}
\end{align*}
Hence, applying Theorem~\ref{thm:bounds}, we have that for the ADP scheme $G_K$, $f(G_K)\geq\beta f(O_K)$ where $\beta$ is related to the above $\epsilon_k$ and $\eta_k$ as given in Theorem~\ref{thm:bounds}. In the next section, we provide some examples of special cases to illustrate this bound.

\section{Examples}\label{sc:V}

\subsection{Rollout}

For the remainder of the paper, assume that $x_1$ is a given state.
Suppose that $\pi_b:\Xc\to\Ac$ is a given policy. Consider the associated ADP where
$W_{k+1}(x_k,g)=\sum_{i=k+1}^K r_i({x}_i,\pi_b({x}_i))$, where ${x}_{k+1}=h_k({x}_k,g)$ and ${x}_{i+1}=h_k({x}_i, \pi_b({x}_i))$ for $ k+1\leq i\leq K-1$.
This ADP method is called \emph{rollout} \cite{BeT97}; the policy $\pi_b$ is called the \emph{base policy}. For rollout, we have
\begin{equation}
\epsilon_k=1-\frac{r_{k+1}(x_{k+1}, g_{k+1})+R_1-R_2}{r_1(x_1, g_1)+\sum_{i=2}^K r_i(\tilde{x}_i, \pi_b(\tilde{x}_i))},
\label{eqn:Repsilon}
\end{equation}
where
\begin{align*}
R_1 &= \sum_{i=k+2}^K r_i(x_i, \pi_b(x_i)), \\
R_2 &= \sum_{i=k+1}^K r_i(\hat{x}_i, \pi_b(\hat{x}_i)), \\
x_{i+1} &= h_i(x_i, g_i) \ \text{for}\ 1\leq i\leq k+1,\\
x_{i+1} &= h_i(x_i, \pi_b(x_i))\ \text{for}\  k+2\leq i\leq K-1,\\
\hat{x}_{k+1} &= h_k(x_k, g_k), \\
\hat{x}_{i+1}&= h_i(\hat{x}_i, \pi_b(\hat{x}_i))\ \text{ for}\ k+1\leq i \leq K-1,\\
\tilde{x}_2 &= h_1(x_1, g_1),\\
\tilde{x}_{i+1} &= h_i(\tilde{x}_i, \pi_b(\tilde{x}_i))\ \text{for}\ 2\leq i\leq K-1.
\end{align*}
Moreover, we have
\begin{equation}
\eta_k=\frac{r_{k+1}(x_{k+1}, o_{k+1})+R_3-R_4}{r_1(x_1,o_{k+1})+\sum_{i=2}^K r_i(\tilde{x}_i, \pi_b(\tilde{x}_i))},
\label{eqn:Reta}
\end{equation}
where 
\begin{align*}
R_3 &= \sum_{i=k+2}^K r_i(x_i, \pi_b(x_i)),\\
R_4 &= \sum_{i=k+1}^K r_i(\hat{x}_i, \pi_b(\hat{x}_i)),\\
x_{i+1} &= h_i(x_i, o_i) \ \text{for}\ 1\leq i\leq k+1,\\
x_{i+1} &= h_i(x_i, \pi_b(x_i))\ \text{for}\  k+2\leq i\leq K-1,\\
\hat{x}_{k+1} &= h_k(x_k, o_k), \\
\hat{x}_{i+1} &= h_i(\hat{x}_i, \pi_b(\hat{x}_i))\ \text{ for}\ k+1\leq i \leq K-1,\\
\tilde{x}_2 &= h_1(x_1, o_{k+1}),\\
\tilde{x}_{i+1} &= h_i(\tilde{x}_i, \pi_b(\tilde{x}_i))\ \text{for}\ 2\leq i\leq K-1.
\end{align*}
We now show that for rollout, the function $f$ is forward monotone with respect to $G_k$, which implies that $\epsilon_k\leq 1$ and $\beta> 0$ (see Remark~1 in Section~\ref{sc:IV}).

\begin{Theorem}
\label{thm:rollout_monotone}
In rollout, $f(G_{k})\geq f(G_{k-1})$ for $k=1,\ldots,K$.
\end{Theorem}

\begin{proof}
We have
\begin{align*}
&f(G_k)-f(G_{k-1})\\
&=f((g_1,\ldots,g_{k-1},g_k))-f((g_1,\ldots,g_{k-1})) \\
&=(\sum_{i=1}^{k-1} r_i(x_i,g_i) + r_k(x_k,g_k) + W_{k+1}(x_k,g_k))- \\
&\quad (\sum_{i=1}^{k-1} r_i(x_i,g_i) + r_k(x_k,\pi_b(x_k))+W_{k+1}(x_{k},\pi_b(x_k)))\\
&= ( r_k(x_k,g_k) + W_{k+1}(x_k,g_k))-\\
&\quad(r_k(x_k,\pi_b(x_k)+W_{k+1}(x_{k},\pi_b(x_k))).
\end{align*}
By (\ref{eqn:5}), we have that 
\[
r_k(x_k,g_k) + W_{k+1}(x_k,g_k)=\max_{g\in\mathcal{A}} \{r_k(x_k,g) + W_{k+1}(x_k,g)\},
\]
which implies that $f(G_k)\geq f(G_{k-1})$. 
\end{proof}

\subsection{Rollout with Optimal Base Policy}

Suppose that the base policy is the optimal policy. In this case, 
the VTG approximation term $W_{k+1}$ is equal to the true VTG.
As pointed out in Section~\ref{sc:III}, the resulting rollout scheme 
is optimal and satisfies Bellman's optimality principle. In this case, 
of course $f(G_K)=f(O_K)$. To illustrate that the bound in Theorem~\ref{thm:bounds} is tight in this case, we will show that $\beta=1$. We do this by showing that $\sum_{i=0}^{K-1} \epsilon_i + \eta_i \leq K$ (see Remark~2 in Section~\ref{sc:IV}). To see this, by (\ref{eqn:cur1}), we have $\epsilon_0=0$ and $\epsilon_k=1$ for $1\leq k\leq K-1$. By (\ref{eqn:cur2}), we have $\eta_0=1$ and $\eta_k=0$ for $1\leq k\leq K-1$. Therefore, $\sum_{i=0}^{K-1}(\epsilon_i+\eta_i)= K$, which implies that $\beta=1$.

\subsection{Myopic Policy}

Consider the special case where $W_{k+1}(\cdot)\equiv 0$ for each $k$. In other words, for each $k$, $g_k \in \argmax_{g\in\Ac} r_k(x_k,g)$. We call this the \emph{myopic policy}. For the myopic policy, we have that $k=0,\ldots, K-1$,
\[
\epsilon_k=1-\frac{r_{k+1}(x_{k+1}, g_{k+1})}{r_1(x_1, g_1)}, 
\]
where $x_{i+1}=h_i(x_i, g_i)$ for $1\leq i\leq k-1$, and
\[
\eta_k=\frac{r_{k+1}(x_{k+1}, o_{k+1})}{r_1(x_1, o_{k+1})},
\]
where $x_{i+1}=h_i(x_i, o_i)$ for $1\leq i\leq k-1$. 
It is clear that because $r_k(\cdot,\cdot)>0$, we have $\epsilon_k<1$, in which case $\beta>0$. In fact, it is easy to check that $f$ is forward monotone with respct to $G_k$ in this case. 

\subsection{Rollout of Myopic Base Policy}

Consider the rollout method where the base policy is the myopic policy defined above. It is well known that the resulting rollout scheme performs at least as well as the myopic policy \cite{BeT97}. Here, we will calculate a bound on the amount by which the rollout scheme outperforms the myopic base policy in terms of $\epsilon_k$ and $\eta_k$. This calculation involves introducing some additional notation (which seems unavoidable).

Let $G_K^M=(g_1^M,\ldots,g_K^M)$ be the myopic strategy and 
$G_K^{RM}=(g_1^{RM},\ldots,g_K^{RM})$ the corresponding rollout strategy.
More specifically, given $g_1^M,\ldots, g_{k-1}^M$, 
\[
g_k^M \in \argmax_{g^M\in\Ac} r_k(x_k^M,g^M) 
\]
where $x_1^M=x_1$ is given and $x_{i+1}^M=h_i(x_i^M, g_i^M)$ for $1\leq i\leq K-1$.
Moreover, the rollout scheme with the myopic base policy is as follows: 
given $g_1^{RM},\ldots, g_{k-1}^{RM}$,
\[
g_k^{RM} \in \argmax_{g^{RM}\in\Ac} \{r_k(x_k^{RM},g^{RM}) +\sum_{i=k+1}^K r_i({x}_i^{RM},\pi_b({x}_i^{RM}))\}
\]
where 
\[
\pi_b({x}_i^{RM}) \in \argmax_{g^{RM}\in \Ac}r_i({x}_i^{RM},g^{RM}),
\]
$x_1^{RM}=x_1$ is given, $x_{i+1}^{RM}=h_i(x_i^{RM}, g_i^{RM})$ for $1\leq i\leq k-1$, and $x_{i+1}^{RM}=h_i(x_i^{RM}, \pi_b(x_i^{RM}))$ for $k\leq i\leq K-1$.

Let $f^M$ and $f^{RM}$ respectively denote the objective functions corresponding to the myopic and rollout (with myopic base policy) strategies.
Then we have that 
\[
f^M((g_1^M,\ldots,g_k^M))=\sum_{i=1}^k r_i(x_i^M,g_i^M)
\]
where $x_{i+1}^M=h_i(x_i^M, g_i^M)$ for $1\leq i\leq k-1$, and $x_1^M=x_1$ is given.
Moreover,
\begin{equation*}
\begin{array}{ll}
f^{RM}((g_1^{RM},\ldots, g_k^{RM})) & \nonumber\\
=\sum_{i=1}^k r_i({x}_i^{RM}, g_i^{RM})+\sum_{i=k+1}^K r_i({x}_i^{RM}, \pi_b({x}_i^{RM}))&\nonumber
\end{array}
\end{equation*}
where ${x}_{i+1}^{RM}=h_i({x}_i^{RM}, g_i^{RM})$ for $1\leq i\leq k$, ${x}_{i+1}^{RM}=h_i({x}_i^{RM}, \pi_b({x}_i^{RM}))$ for $k+1\leq i \leq K-1$, and ${x}_1^{RM}=x_1$ is given.

%We claim that
%$f^{RM}(G_K^{RM})\geq f^{RM}(G_{K-1}^{RM})\geq \cdots\geq f^{RM}(G_1^{RM})\geq f^M(G_K^M)\geq f^M(G_{K-1}^M)\geq\cdots\geq f^M(G_1^M)$, which is shown as follows.
%By Theorem~\ref{thm:rollout_monotone}, we have $f^{RM}(G_K^{RM})\geq f^{RM}(G_{K-1}^{RM})\geq \cdots\geq f^{RM}(G_1^{RM})$ and by monotonicity of $f$ for the mypic policy, $f^M(G_K^M)\geq f^M(G_{K-1}^M)\geq\cdots\geq f^M(G_1^M)$. Hence, it suffices to prove that $f^{RM}(G_1^{RM})\geq f^M(G_K^M)$.

We claim that $f^{RM}(G_1^{RM})\geq f^M(G_K^M)$. Tos see this, for the myopic policy, we have 
\[
g_k^M \in \argmax_{g^M\in\mathcal{A}}r_k(x_k^M,g^M)
\]
for $k=1,\ldots,K$. For rollout with the myopic base policy, we have 
\begin{align*}
g_1^{RM} & \in \argmax_{g^{RM}\in\mathcal{A}}\{r_1({x}_1^{RM},g^{RM})\\
& \mbox{}+r_2({x}_2^{RM},\pi_b({x}_2^{RM}))\\
& \mbox{}+\cdots+r_K({x}_K^{RM},\pi_b({x}_K^{RM}))\}.
\end{align*}
Because $\pi_b({x}_i^{RM}) \in \argmax_{g^{RM}\in \Ac}r_i({x}_i^{RM},g^{RM})$ and 
$x_1^M={x}_1^{RM}$, we have that 
\begin{align*}
&r_1(x_1^{RM},g_1^{RM})+r_2({x}_2^{RM},\pi_b({x}_2^{RM}))\\
&\mbox{}+\cdots+r_K({x}_K^{RM},\pi_b({x}_K^{RM})) \\
&\geq r_1(x_1^M,g_1^M)+r_2({x}_2^M,g_2^M)\\
&\mbox{}+\cdots+r_K({x}_K^M,g_K^M),
\end{align*}
which means that $f^{RM}(G_1^{RM})\geq f^M(G_K^M)$, as desired.

Combining (\ref{eqn:GK1}), (\ref{eqn:G1O}), and the inequality
$f^{RM}(G_1^{RM})\geq f^M(G_K^M)$, we have 
\begin{align*}
f^{RM}(G_K^{RM})-f^M(G_K^M)
&\geq ( \sum_{i=1}^{K-1}(1-\epsilon_i))f^{RM}(G_1^{RM})\\
&\geq\frac{\sum_{i=1}^{K-1}(1-\epsilon_i)}{\sum_{i=0}^{K-1}(\eta_i)}f^{RM}(O_K),
\end{align*}
which provides a bound on the amount by which the rollout scheme outperforms the myopic base policy.

\section{Conclusion}\label{sc:VI}

We have developed a systematic approach to deriving guaranteed bounds for approximate dynamic programming (ADP) schemes in optimal control problems. The approach is to formulate a string-submodular optimization problem for which the optimal strategy is the optimal control solution, and the greedy strategy is the ADP solution. Using this approach, we have shown that any ADP solution achieves a performance that is at least a factor of $\beta$ of the performance of the optimal control solution (satisfying Bellman's optimality principle). The factor $\beta$ depends on the specific ADP scheme. We have explicitly characterized this dependence and we have illustrated the the applicability of our bounding technique to a few examples of ADP schemes, including the popular rollout method.

%We consider various kinds of policies for solution of  the deterministic control problem. We have introduced the concepts of forward curvature and elemental forward curvature for functions defined on the space of strings.  These have enabled us to prove performance bounds for general approximate dynamic programming solution schemes. The curvatures are calculated for both optimal policy rollout and myopic base policy rollout. We have investigated how rollout of myopic base policy compares with the myopic policy, expressed in terms of the optimal solution.

%\clearpage \pagenumbering{arabic}
%\def\thepage{RF-\arabic{page}}
%%\bibliographystyle{IEEEtranS}
%%\bibliography{bib1}

\bibliographystyle{IEEEbib}

\end{document}